\documentclass[11pt,reqno]{amsart}

\newtheorem{proposition}{Proposition}[section]

\newtheorem{corollary}[proposition]{Corollary}
\newtheorem{theorem}[proposition]{Theorem}

\theoremstyle{definition}
\newtheorem{definition}[proposition]{Definition}
\newtheorem{example}[proposition]{Example}

\newcommand{\thlabel}[1]{\label{th:#1}}
\newcommand{\thref}[1]{Theorem~\ref{th:#1}}
\newcommand{\selabel}[1]{\label{se:#1}}
\newcommand{\seref}[1]{Section~\ref{se:#1}}

\newcommand{\prlabel}[1]{\label{pr:#1}}
\newcommand{\prref}[1]{Proposition~\ref{pr:#1}}
\newcommand{\colabel}[1]{\label{co:#1}}
\newcommand{\coref}[1]{Corollary~\ref{co:#1}}

\def\ot{\otimes}

\def\CC{{\mathbb C}}

\newcommand{\Cc}{\mathcal{C}}

\def\*C{{}^*\hspace*{-1pt}{\Cc}}
\def\text#1{{\rm {\rm #1}}}

\input xy
\xyoption {all} \CompileMatrices

\usepackage{amssymb}
\usepackage{color,amssymb,graphicx,amscd,amsmath}
\usepackage[colorlinks,urlcolor=blue,linkcolor=blue,citecolor=blue]{hyperref}

\begin{document}
\title[The maximal dimension of unital subalgebras of the matrix algebra]
{The maximal dimension of unital subalgebras of the matrix
algebra}

\author{A. L. Agore}
\address{Faculty of Engineering, Vrije Universiteit Brussel, Pleinlaan 2, B-1050 Brussels, Belgium \textbf{and} ''Simion Stoilow'' Institute of Mathematics of the Romanian Academy, P.O. Box 1-764, 014700 Bucharest, Romania} \email{ana.agore@vub.ac.be and
ana.agore@gmail.com}

\thanks{The author is Postdoctoral Fellow of the Fund for Scientific
Research-Flanders (Belgium) (F.W.O. Vlaanderen). This research is
part of the grant no. 88/05.10.2011 of the Romanian National
Authority for Scientific Research, CNCS-UEFISCDI}

\subjclass[2010]{16D70, 16Z05} \keywords{matrix (co)algebra,
parabolic subalgebras of matrix algebras, parabolic coideals of
matrix coalgebras, maximal dimension of subalgebras}

%\maketitle

\maketitle

\begin{abstract}
We prove that over a field of characteristic zero the maximal
dimension of a proper unital subalgebra in the $n \times n$ matrix
algebra is $n^2 - n + 1$ and furthermore this upper bound is
attained for the so-called parabolic subalgebras. We also
investigate the corresponding notion of parabolic coideals for
matrix coalgebras and prove that the minimal dimension of a
non-zero coideal of the matrix coalgebra ${\mathcal M}^n (k)$ is
$n-1$.
\end{abstract}

\section*{Introduction}

The problem of finding the minimal/maximal dimension of proper
subobjects (possibly satisfying some extra assumptions) of a given
mathematical object is a natural one and it was studied in various
settings. The first result in this direction was proved by Schur
\cite{Sch} and asserts that $[\frac{n^{2}}{4}] + 1$ is an upper
bound for the dimensions of commutative subalgebras of
$\mathcal{M}_{n}(\CC)$. Later on, Jacobson \cite{J} showed that
the same bound is valid over any field $K$. An elegant and simpler
proof of Schur's theorem was given by Mirzakhani in \cite{MM}.
These results allowed for the introduction of the \emph{Schur
invariant} in Lie algebra theory as the maximal dimension of
abelian subalgebras of a given Lie algebra. A vast literature
emerged from the study of this invariant which plays an important
role in many aspects of Lie algebra theory, see for example
\cite{BCM, NePo} and the references therein. To give just one
example, the Schur invariant has been completely determined for
semisimple Lie algebras by Malcev in \cite{Mal}. In this note we
address the problem of finding the maximal dimension of
subalgebras of the matrix algebra $\mathcal{M}_{n}(K)$ over a
field $K$ of characteristic zero. Our main result is \thref{imp}
which gives an upper bound for the dimension of subalgebras in
$\mathcal{M}_{n}(K)$ and it proves that this maximal dimension is
attained for the so-called parabolic subalgebras \cite{WPW}, which
are the associative algebra counterpart of the well-known
parabolic Lie subalgebras (see \cite{WY, W} for more details). A
subalgebra $\mathcal{A}$ of $\mathcal{M}_{n}(K)$ is called
parabolic if it is similar to an algebra consisting of all
matrices having non-overlapping blocks of $n_{i} \times n_{i}$
matrices on the diagonal, $i = 1, 2, ..., s$, with non-zero
entries only in these blocks or above them, where $(n_{1}, n_{2},
..., n_{s})$ is a partition of $n$. As a consequence of
\thref{imp} we prove in \coref{maxi} that the maximal dimension of
a proper unital subalgebra of $\mathcal{M}_{n}(K)$ is equal to
$n^2 - n + 1$. A related result was obtained in \cite{am-2013c}
where it was proved that $\mathcal{M}_{n}(K)$ has no subalgebras
of codimension equal to $1$. A somewhat similar problem was
considered recently in \cite{BCM, CT} in the context of Lie
algebras and in \cite{M} for Lie algebras as well as for
associative algebras. In the above mentioned papers the authors
are interested in finding bounds for the dimensions of abelian
subalgebras or abelian ideals.

The paper ends with a brief discussion on matrix coalgebras (see
\seref{prel} for the definition and further details). We introduce
the notion of parabolic coideals for matrix coalgebras and derive
some of their features based on the duality between the matrix
algebra and the matrix coalgebra. As an application, we prove that
the minimal dimension of a non-zero coideal of the matrix
coalgebra $\mathcal{M}^{n}(K)$ is $n-1$.

\section{Preliminaries}\selabel{prel}
Unless otherwise stated, all vector spaces, linear maps, algebras
and coalgebras are over an arbitrary field $K$. In some specific
cases we will assume that $K$ is algebraically closed of characteristic zero. For a vector space $V$ we denote by $V^{*} :
= {\rm Hom}_K (V, \, K)$ the dual of $V$. If $S$ is a subset of $V$ we
set $S^{\perp} : = \{v \in V^{*} ~|~ v(S) = 0\}$. Similarly, if
$Y$ is a subset of $V^{*}$ we denote $Y^{\perp} : = \{x \in V ~|~
f(x) = 0,\, {\rm for} \, {\rm all}\, f \in Y\}$. If $X$ is a
subspace of the vector space V then the codimension of $X$ is
${\rm dim}_{K}\,(V/X)$.

By an algebra (coalgebra) we mean an associative (coassociative),
unital (counital)  algebra (coalgebra) over $K$. If $A$ is an
algebra, ${\rm rad}\, A$ stands for the Wedderburn-Artin radical
of $A$. Wedderburn's main theorem (\cite[Theorem 2.17]{K}) is a
fundamental result on the structure of finite dimensional
algebras, being the associative algebra counterpart of the famous
Levi decomposition of a Lie algebra: it asserts that for any
finite dimensional algebra $A$ over a field $K$ of characteristic
zero, there exists a semisimple subalgebra $S$ of $A$ isomorphic
to $A / {\rm rad}\, A$ such that $A = S \oplus {\rm rad}\, A$.
Moreover, over an algebraically closed field any semisimple
finite-dimensional algebra is a direct sum of matrix algebras over
the same field. We refer the reader to \cite[Chapter II]{K} for a
detailed account on the Wedderburn-Artin theory. For an arbitrary
integer $n \geq 2$ we let $\mathcal{M}_{n}(K)$ denote the algebra
of $n \times n$ matrices over the field $K$ while
$\mathcal{U}_{n}(K)$ and $\overline{\mathcal{U}_{n}(K)}$ stand for
the set of all upper triangular matrices in $\mathcal{M}_{n}(K)$,
respectively the set of all strictly upper triangular matrices in
$\mathcal{M}_{n}(K)$. We denote by $e_{i,\,j} \in
\mathcal{M}_{n}(K)$ the matrix having $1$ in the $(i,j)^{th}$
position and zeros elsewhere. If $\mathcal{A}$ is a subset of
$\mathcal{M}_{n}(K)$, then we denote $C \mathcal{A} C^{-1} = \{C A
C^{-1} ~|~ A \in \mathcal{A}\}$.

If $A$ is a finite dimensional algebra then $A^{*}$ has a natural
coalgebra structure called the dual coalgebra on $A$.  In the same
manner, if $C$ is a (not necessarily finite dimensional) coalgebra
then $C^{*}$ inherits an algebra structure called the dual algebra
on $C$. $\mathcal{M}^{n}(K)$ denotes the matrix coalgebra, i.e.
$\mathcal{M}^{n}(K) = \mathcal{M}_{n}(K)$ as vector spaces with
the coalgebra structure given for all $i$, $j = 1, 2, ..., n$ by:
$$\Delta: \mathcal{M}^{n}(K) \to \mathcal{M}^{n}(K) \otimes \mathcal{M}^{n}(K), \qquad \Delta(e_{i,\,j}) = \sum_{k=1}^{n} e_{i,\,k} \otimes e_{k,\, j}$$
$$\varepsilon: \mathcal{M}^{n}(K) \to K, \qquad \varepsilon(e_{i,\,j}) = \delta_{i,\,j}$$
It is well-known that $\mathcal{M}_{n}(K)^{*} $ and
$\mathcal{M}^{n}(K)$ are isomorphic as coalgebras. Similarly,
$\mathcal{M}^{n}(K)^{*}$ and $\mathcal{M}_{n}(K)$ are isomorphic
as algebras. For all unexplained notions or results from the
theory of coalgebras we refer the reader to \cite{BW, SW}.

We also make use of a well-known result on the dimension of
subspaces of nilpotent matrices which first appeared in \cite{G}.
An improved version of this result appears in \cite{S}, where the
cardinality assumption on the field $K$ is removed. More
precisely, the result in \cite{S} states that if $\mathcal{S}$ is
a subspace of the vector space of $n \times n$ matrices over an
arbitrary field $K$ and $\mathcal{S}$ consists of nilpotent
matrices, then the maximal dimension of $\mathcal{S}$ is
$\frac{n(n-1)}{2}$. Moreover, if equality holds then $\mathcal{S}$
is conjugate to $\overline{\mathcal{U}_{n}(K)}$.

\section{Parabolic subalgebras of matrix algebras}

By analogy with Lie algebra theory (\cite{WY}), a proper
subalgebra $\mathcal{A}$ of the matrix algebra
$\mathcal{M}_{n}(K)$ is called \emph{parabolic} \cite{WPW} if it
is similar to an algebra which contains $\mathcal{U}_{n}(K)$, i.e.
there exists an invertible matrix $U$ such that
$\mathcal{U}_{n}(K) \subseteq U \mathcal{A} U^{-1}$. As we will
see, this concept will play an important role in determining the
maximal dimension of a subalgebra in a matrix algebra. It was
proved in \cite{WPW} that $\mathcal{A}$ is a parabolic subalgebra
of $\mathcal{M}_{n}(K)$ if and only if there exists a set of
positive integers $n_{1}$, $n_{2}$, ..., $n_{s}$ with
$\sum_{i=1}^{s} n_{i} = n$ such that $\mathcal{A}$ is similar to
the algebra of all matrices having non-overlapping blocks of
$n_{i} \times n_{i}$ matrices on the diagonal with non-zero
entries only in these blocks or above them. Moreover, it is
straightforward to see that ${\rm dim}\, A = \frac{n^{2}}{2} +
\sum_{i=1}^{s}\frac{n_{i}^{2}}{2}$.

\begin{definition}
A parabolic subalgebra $\mathcal{A}$ of $\mathcal{M}_{n}(K)$
determined by the set of positive integers $n_{1}$, $n_{2}$, ...,
$n_{s}$ will be called a parabolic subalgebra of type $(n_{1},
n_{2}, ..., n_{s})$. If $s=2$ then $\mathcal{A}$ will be called a
maximal parabolic subalgebra.
\end{definition}

Our next result shows that a maximal parabolic subalgebra is in
fact a maximal proper subalgebra of $\mathcal{M}_{n}(K)$.

\begin{proposition}\prlabel{1.2}
The maximal parabolic subalgebras are maximal proper subalgebras
of $\mathcal{M}_{n}(K)$.
\end{proposition}

\begin{proof}
Consider $\mathcal{A}$ to be a parabolic subalgebra having on the
diagonal two blocks $U$ and $V$ of dimensions $l \times l$ and
respectively $(n-l) \times (n-l)$ with $l \geq 1$ and assume there
exists a subalgebra $\mathcal{T}$ of $\mathcal{M}_{n}(K)$ such
that $\mathcal{A} \subset \mathcal{T}$, $\mathcal{A} \neq
\mathcal{T}$. The elements of $\mathcal{T}$ are linear
combinations of the matrix units $e_{i, \, j}$ and there must be an element $x \in \mathcal{T} - \mathcal{A}$ that has a non-zero entry $\alpha$ in a position $(i,\, j)$ with $i >  l$ and $j \leqslant l$. Then $e_{i,\, j} \in \mathcal{T} $ since $e_{i,\,j} = \alpha^{-1} \, e_{i,\,i} \, x \, e_{j,\,j}$. Since $i > l$, $e_{m,\,i} \in \mathcal{A}$ for all $m$ and hence $e_{m, \, j} = e_{m,\, i} \, e_{i,\,j}$ belongs to $\mathcal{T}$ for all $m$. Similarly $e_{j,\,k} \in \mathcal{A}$ for all $k$ and so $e_{i,\,k}  = e_{i,\,j}\, e_{j,\,k} \in \mathcal{T}$ for all $k$. Thus if $e_{i,\, j}$ belongs to $\mathcal{T}$, then so do all the matrix units from row $i$ and column $j$. Using the same reasoning for all these elements we can conclude that $\mathcal{T} = \mathcal{M}_{n}(K)$ which finishes the proof.

\end{proof}

\begin{example}
The parabolic subalgebras of $\mathcal{M}_{3}(K)$ are similar to
the following subalgebras:
$$
\begin{pmatrix} K & K & K \\
0 & K & K\\
0 & 0 & K \end{pmatrix} \qquad \begin{pmatrix} K & K & K \\
K & K & K\\
0 & 0 & K \end{pmatrix} \qquad \begin{pmatrix} K & K & K \\
0 & K & K\\
0 & K & K \end{pmatrix}
$$
The last two algebras are maximal proper subalgebras of
$\mathcal{M}_{3}(K)$.
\end{example}

If $K$ is an algebraically closed field of characteristic zero and
$\mathcal{A}$ is a finite dimensional $K$-algebra then by
Wedderburn's theorem \cite{K}, $\mathcal{A}$ has a semisimple part
which is a direct sum of matrix algebras over $K$ of dimensions
$n_{1}^{2}$, $n_{2}^{2}$, ..., $n_{s}^{2}$. In order for
$\mathcal{A}$ to be a subalgebra of $\mathcal{M}_{n}(K)$ we need
to have $\sum_{i=1}^{s} n_{i} \leq n$. However, since we are
trying to maximize the dimension of $\mathcal{A}$ we will assume
that $\sum_{i=1}^{s} n_{i} = n$.

\begin{theorem}\thlabel{imp}
Let $K$ be an algebraically closed field of characteristic zero
and $\mathcal{A}$ a subalgebra of $\mathcal{M}_{n}(K)$ whose
semisimple part $S$ is a direct sum of matrix algebras over $K$ of
dimensions $n_{1}^{2}$, $n_{2}^{2}$, ..., $n_{s}^{2}$ with
$\sum_{i=1}^{s} n_{i} = n$. Then ${\rm dim}\, \mathcal{A} \leq
\frac{n^{2}}{2} + \sum_{i=1}^{s}\frac{n_{i}^{2}}{2}$. Moreover, if
equality holds then $A$ is a parabolic subalgebra of
$\mathcal{M}_{n}(K)$.
\end{theorem}

\begin{proof}
By Wedderburn's main theorem \cite[Theorem 2.17]{K} we have
$\mathcal{A} = S \oplus {\rm rad}\, \mathcal{A}$. Let ${\rm dim}\,
\mathcal{A} = d$ and ${\rm dim}\, {\rm rad}\, \mathcal{A} = r$. It
follows that $\sum_{i=1}^{s} n_{i}^{2} + r = d$. Remark that each
$n_{i} \times n_{i}$ matrix subalgebra of $S$ contains its
nilpotent subalgebra of strictly upper triangular matrices. This
implies that $\mathcal{A}$ contains, all together, a subalgebra of
nilpotent matrices of dimension $\sum_{i=1}^{s}\frac{n_{i}(n_{i} -
1)}{2} + r$. Using Gerstenhaber's result (\cite{G}) on subspaces
of nilpotent matrices we must have
$\sum_{i=1}^{s}\frac{n_{i}(n_{i} - 1)}{2} + r \leq
\frac{n(n-1)}{2}$. Since $\sum_{i=1}^{s} n_{i}^{2} + r = d$ we get
$d \leq \frac{n(n-1)}{2} + \sum_{i=1}^{s}\frac{n_{i}(n_{i} +
1)}{2}$.  The conclusion now follows by using $\sum_{i=1}^{s}
n_{i} = n$.

Suppose now that equality holds. It follows that $r =
\frac{n^{2}}{2} - \sum_{i=1}^{s}\frac{n_{i}^{2}}{2}$ and so
$\mathcal{A}$ contains a subalgebra of nilpotent matrices of
dimension $\frac{n(n-1)}{2}$. Gerstenhaber's result implies that
any such subspace of $\mathcal{M}_{n}(K)$ is conjugate to
$\overline{\mathcal{U}_{n}(K)}$. Since for any invertible matrix
$C \in \mathcal{M}_{n}(K)$ the map $u: \mathcal{A} \to C
\mathcal{A} C^{-1}$ which takes any $A \in \mathcal{A}$ to
$CAC^{-1}$ is an algebra isomorphism we may assume without loss of
generality that $\mathcal{A}$ contains
$\overline{\mathcal{U}_{n}(K)}$. However, by looking at dimensions
one sees that $\overline{\mathcal{U}_{n}(K)}$ is not all of
$\mathcal{A}$. Now since the semisimple part of $\mathcal{A}$ is a
direct sum of matrix algebras of dimensions $n_{1}$, $n_{2}$, ...,
$n_{s}$ with $\sum_{i=1}^{s} n_{i} = n$ it follows that each
$e_{ii} \in \mathcal{A}$ for all $i = 1, 2, ..., n$. Therefore,
$\mathcal{A}$ contains $\mathcal{U}_{n}(K)$ and the proof is now
finished.
\end{proof}

\begin{proposition}\prlabel{1.3}
Let $K$ be an algebraically closed field of characteristic zero.
Then the proper subalgebras of maximum dimension in
$\mathcal{M}_{n}(K)$ are the parabolic subalgebras of type $(1,
n-1)$ and respectively $(n-1, 1)$.
\end{proposition}

\begin{proof}
By \thref{imp} we know that the upper bound for the dimension of a
subalgebra in $\mathcal{M}_{n}(K)$ is attained when the subalgebra
is parabolic. In what follows we will prove that this bound is
maximal precisely when the subalgebra is parabolic of type $(1,
n-1)$ and respectively $(n-1, 1)$. Since the dimension of a
parabolic subalgebra of $\mathcal{M}_{n}(K)$ is $\frac{n^{2}}{2} +
\sum_{i=1}^{s}\frac{n_{i}^{2}}{2}$ in order to maximize its
dimension we need to have $\sum_{i=1}^{s} n_{i}^{2}$ as large as
possible. We denote by $S := n_{1}^{2} + n_{2}^{2} + ... +
n_{i}^{2} + (n - n_{1} - n_{2} - ... - n_{i})^{2}$ and by $S' :=
n_{1}^{2} + n_{2}^{2} + ... + n_{i-1}^{2} + (n - n_{1} - n_{2} -
... - n_{i-1})^{2}$, where $n_{t}$ are positive integers for all
$t = 1, 2, ..., i$, such that $n_{1} + n_{2} + ... + n_{i} < n$.
The proof will be finished once we show that $S < S'$. Indeed,
this follows by noticing that $S = S' - 2n_{i}\,(n - n_{1} - n_{2}
- ... - n_{i}) < S'$. Therefore we need to consider parabolic
subalgebras of type $(l, n-l)$, with $l \geq 1$. Then
$\sum_{i=1}^{s} n_{i}^{2} = l^{2} + (n-l)^{2}$. Now it can easily
be seen that for $ 1< l < n-1$ we have $l^{2} + (n-l)^{2} < 1 +
(n-1)^{2}$ and the conclusion follows.
\end{proof}

\begin{corollary}\colabel{maxi}
Let $K$ be a field of characteristic zero. Then the maximal dimension of a proper subalgebra of the matrix algebra
$\mathcal{M}_{n}(K)$ is $n^2 - n + 1$.
\end{corollary}
\begin{proof}
If $K$ is an algebraically closed field of characteristic zero then the assertion follows from \prref{1.3}. We will prove that the algebraically closed assumption on the field $K$ can be dropped by simply extending the coefficients to the algebraic closure $\overline{K}$ of $K$. Indeed, let $A$ be a $K$-subalgebra  of $\mathcal{M}_{n}(K)$. We have:
\begin{eqnarray*}
\overline{K} \simeq \overline{K} \ot_{K} K \subset \overline{K} \ot_{K} A \subseteq \overline{K} \ot_{K} \mathcal{M}_{n}(K) \simeq \mathcal{M}_{n}(\overline{K})
\end{eqnarray*}
where the last isomorphism follows from \cite[Lemma 7.130]{rotman}. Therefore, $\overline{K} \ot_{K} A$ is a $\overline{K}$-subalgebra of $\mathcal{M}_{n}(\overline{K})$ and since ${\rm dim}_{K} (A) = {\rm dim}_{\overline{K}}\, (\overline{K} \ot_{K} A)$ the conclusion follows.
\end{proof}

We end the paper by looking at the matrix coalgebra case. It is
well-known that $X \subseteq \mathcal{M}^{n}(K)$ is a coideal if
and only if $X^{\perp}$ is a subalgebra of $\mathcal{M}_{n}(K)$
(see \cite[Proposition 1.4.6]{SW} for a more general statement).
In light of this bijection we introduce the following:

\begin{definition}
A coideal $X \subseteq \mathcal{M}^{n}(K)$ will be called
parabolic if $X^{\perp}$ is a parabolic subalgebra of
$\mathcal{M}_{n}(K)$. If $X^{\perp}$ is a parabolic subalgebra of
type $(n_{1}, \, n_{2}, ..., n_{s})$ then $X$ will be called a
parabolic coideal of type $(n_{1}, \, n_{2}, ..., n_{s})$ as well.
If $s=2$ then $X$ will be called a minimal parabolic coideal.
\end{definition}

The parabolic coideals of $\mathcal{M}^{n}(K)$ can be
characterized as follows:

\begin{proposition}
Let $X$ be a parabolic coideal of type $(n_{1}, \, n_{2}, ...,
n_{s})$ of the matrix coalgebra $\mathcal{M}^{n}(K)$. Then $X$ is
the coideal of all matrices having non-overlapping blocks of
$n_{i} \times n_{i}$ matrices on the diagonal with non-zero
entries only below these blocks.
\end{proposition}
\begin{proof}
If $\mathcal{A}$ is a subalgebra of $\mathcal{M}^{n}(K)^{*} \simeq
\mathcal{M}_{n}(K)$ (isomorphism of algebras) then
$\mathcal{A}^{\perp}$ is a coideal in $\mathcal{M}^{n}(K)$. The
conclusion follows by a straightforward computation.
\end{proof}

\begin{proposition}
The minimal parabolic coideals are minimal proper coideals of
$\mathcal{M}^{n}(K)$.
\end{proposition}
\begin{proof}
Let $X$ be a minimal parabolic coideal of type $(l, n-l)$ and
assume there exists a coideal $Y$ of $\mathcal{M}^{n}(K)$ such
that $Y \subset X$, $Y \neq X$. Then $X^{\perp} \subset Y^{\perp}$
and $X^{\perp}$ is a maximal parabolic subalgebra of
$\mathcal{M}_{n}(K)$. Using \prref{1.2} we obtain $Y^{\perp} =
\mathcal{M}_{n}(K)$. This implies $Y = \{0\}$ and the proof is
finished.
\end{proof}

%If $V$ is a finite dimensional vector space, then we have
%$(S^{\perp})^{\perp} = S$ and $(T^{\perp})^{\perp} = T$ for all
%subsets $S \subseteq V$ and $T \subseteq V^{*}$.

\begin{proposition}
Let $K$ be an algebraically closed field of characteristic zero.
Then the non-zero coideals of minimal dimension in
$\mathcal{M}^{n}(K)$ are those parabolic coideals $X$ for which
$X^{\perp}$ is a parabolic subalgebra of type $(1, n-1)$ and
respectively $(n-1, 1)$.
\end{proposition}
\begin{proof}
For any finite dimensional vector space $V$, and any subspace $X$
of $V^{*}$ we have ${\rm dim }_{K}\, X^{\perp} = {\rm dim}_{K}\,
V^{*}/X$. Therefore, for any coideal $X$ in $\mathcal{M}^{n}(K)$
we have ${\rm dim}_{K}\, X^{\perp} = {\rm dim}_{K}\,
\mathcal{M}^{n}(K) /X$. Collaborating this result with the
bijection between the coideals of $\mathcal{M}^{n}(K)$ and the
subalgebras of $\mathcal{M}_{n}(K)$ (\cite{SW}), it follows that a
coideal $X$ in $\mathcal{M}^{n}(K)$ has maximal codimension
precisely when the subalgebra $X^{\perp}$ of $\mathcal{M}_{n}(K)$
has maximal dimension. The conclusion now follows by \prref{1.3}.
\end{proof}

\begin{corollary}
Let $K$ be a field of characteristic zero. Then the minimal dimension of a
non-zero coideal in $\mathcal{M}^{n}(K)$ is $n - 1$.
\end{corollary}

\end{document}